\documentclass[notitlepage,reqno,12pt]{amsart}

\pdfoutput=1

\usepackage{amsmath}
\usepackage{amssymb,amsfonts,amsthm}
\usepackage[foot]{amsaddr}
\usepackage{graphicx}
\usepackage{bm}
\usepackage{subfigure}
\usepackage[square,numbers,sort&compress]{natbib}
\usepackage{hyperref}
\usepackage[letterpaper,hmarginratio=1:1]{geometry}
\graphicspath{{figs/}}
\usepackage{scalerel,stackengine}
\usepackage{dsfont}
\usepackage{tikz}
\usetikzlibrary{matrix}
\usepackage{thmtools}
\usepackage{thm-restate}
\usepackage[capitalize]{cleveref}
\usepackage{blkarray, bigstrut}
\usepackage{booktabs}
\usepackage{colortbl}
\usepackage[normalem]{ulem}  

\numberwithin{equation}{section}

\theoremstyle{definition}
\newtheorem{theorem}{Theorem}
\newtheorem{lemma}{Lemma}
\newtheorem{corollary}{Corollary}
\newtheorem*{remark}{Remark}  
\declaretheorem[name=Example,qed={\lower-0.3ex\hbox{$\triangle$}}]{Example}
\newtheorem{definition}{Definition}
\newtheorem{property}{Property}

\crefname{property}{Property}{Properties}
\Crefname{property}{Property}{Properties}

\newcommand{\mathnotation}[2]{\newcommand{#1}{\ensuremath{#2}}}

\newcommand{\efrac}[2]{#1 / #2}             

\newcommand{\jlt}[1]{}
\newcommand{\bwo}[1]{}
\stackMath
\newcommand\reallywidehat[1]{%
\savestack{\tmpbox}{\stretchto{%
		\scaleto{%
			\scalerel*[\widthof{\ensuremath{#1}}]{\kern-.6pt\bigwedge\kern-.6pt}%
			{\rule[-\textheight/2]{1ex}{\textheight}}
		}{\textheight}%
	}{0.5ex}}%
\stackon[1pt]{#1}{\tmpbox}%
}
\hypersetup{
colorlinks,
linkcolor={red!50!black},
citecolor={blue!50!black},
urlcolor={blue!80!black}
}

\DeclareMathOperator{\sgn}{sgn}
\newcommand\x{\times}

\newcommand{\ra}[1]{\renewcommand{\arraystretch}{#1}}

\newcommand{\proofsubpart}[1]{\medskip\noindent\textsl{#1:}}
\newcommand{\st}{\,\vert\,}

\let\originalleft\left
\let\originalright\right
\renewcommand{\left}{\mathopen{}\mathclose\bgroup\originalleft}
\renewcommand{\right}{\aftergroup\egroup\originalright}
\renewcommand{\l}{\left}
\renewcommand{\r}{\right}
\mathnotation{\pd}{\partial}
\mathnotation{\ldef}{\mathrel{\raisebox{.069ex}{:}\!\!=}}
\mathnotation{\rdef}{\mathrel{=\!\!\raisebox{.069ex}{:}}}
\mathnotation{\dint}{\,{\mathrm{d}}}        
\mathnotation{\ee}{\mathrm{e}}              
\mathnotation{\imi}{\mathrm{i}}             

\mathnotation{\Sobo}{\dot{H}}             
\newcommand{\norm}[1]{\l\lVert#1\r\rVert}    

\newcommand{\pair}[2]{\savg{#1\,,\,#2}}
\newcommand{\savg}[1]{\l\langle #1\r\rangle} 
\mathnotation{\xc}{x}               
\mathnotation{\xv}{\bm{\xc}}        
\mathnotation{\fw}{f}               
\mathnotation{\Lone}{L^1}           
\mathnotation{\Ltwo}{L^2}           
\mathnotation{\qq}{q}               
\mathnotation{\kc}{k}               
\mathnotation{\kv}{{\bm{\kc}}}      
\mathnotation{\km}{\kc}             
\mathnotation{\lv}{{\bm{\lc}}}      
\mathnotation{\lc}{{\ell}}           
\mathnotation{\Lsc}{L}               
\mathnotation{\Cu}{c}               
\mathnotation{\NK}{K}               
\mathnotation{\gt}{g}               
\mathnotation{\T}{T}                
\mathnotation{\sdim}{d}             
\mathnotation{\Vol}{\mathbb{T}^\sdim}         
\newcommand{\aaa}{a}					
\newcommand{\bbb}{b}					
\newcommand{\rec}{$q$-recurrent}        		   	 
\newcommand{\tran}{$q$-transient}       			 
\newcommand{\unif}{r}        					 
\newcommand{\indiv}{\lambda}        				 
\newcommand{\indivBigOh}{\varrho}        				 
\newcommand{\Hq}{\Sobo^\qq}    						
\newcommand{\mixnorm}[1]{\l\lVert#1\r\rVert_{\Sobo^{-\qq}}}   		
\newcommand{\typeone}{$(q,\rate)$-recurrent}        
\newcommand{\typetwo}{$(q,\rate)$-transient}        
\newcommand{\rate}{h}       			 
\newcommand{\constant}{c}       			 
\newcommand{\ratio}[1]{\l( \frac{\rate(T_{#1})}{\mixnorm{f^{T_{#1}}}}	\r)^2}       			 
\newcommand{\BigOh}[1]{\mathcal{O}\l( #1 \r)}       
\newcommand{\littleOh}[1]{o\l( #1 \r)}     		 
\newcommand{\BigOhText}{Big-O}  					 
\newcommand{\littleOhText}{Little-O}       

\begin{document}

\title[On mix-norms and the rate of decay of correlations]{On mix-norms and the rate of \\ decay of correlations}

\author{Bryan W. Oakley$^1$}
\author{Jean-Luc Thiffeault$^1$}
\address{$^1$Department of Mathematics, University of Wisconsin --
	Madison, WI 53706, USA}
\author{Charles R. Doering$^2$}
\address{$^2$Center for the Study of Complex Systems, Department of Mathematics
	and Department of Physics, University of Michigan, Ann Arbor, MI 48109, USA}
\email{\href{mailto:boakley@wisc.edu}{boakley@wisc.edu}}
\email{\href{mailto:jeanluc@math.wisc.edu}{jeanluc@math.wisc.edu}}
\email{\href{mailto:doering@umich.edu}{doering@umich.edu}}

\date{\today}

\begin{abstract}
Two quantitative notions of mixing are the decay of correlations and the decay of a mix-norm --- a negative Sobolev norm --- and the intensity of mixing can be measured by the rates of decay of these quantities.
From duality, correlations are uniformly dominated by a mix-norm; but can they decay asymptotically faster than the mix-norm?
We answer this question by constructing an observable with correlation that comes arbitrarily close to achieving the decay rate of the mix-norm.
Therefore the mix-norm is the sharpest rate of decay of correlations in both the uniform sense and the asymptotic sense.
Moreover, there exists an observable with correlation that decays at the same rate as the mix-norm if and only if the rate of decay of the mix-norm is achieved by its projection onto low-frequency Fourier modes.
In this case, the function being mixed is called \emph{\rec}; otherwise it is \emph{\tran}.
We use this classification to study several examples and raise questions for future investigations.
\end{abstract}

\maketitle


\section{Introduction}

Consider a spatially-periodic mean-zero function $f^t(\xv) = f(t, \xv)$ bounded uniformly in $L^2(\Vol)$ for all $t>0$.
For example, $f(t, \xv)$ might be a solution to the advection-diffusion equation
\begin{equation}\label{advDiffEqn}
  \frac{\pd f}{\pd t} + \bm{u} \cdot \nabla f = D \Delta f,
\end{equation}
with~$f^0 \in L^2(\Vol)$ and smooth divergence-free velocity field~$\bm{u}(t,\bm{x})$.
We may also consider~$D=0$ in~\cref{advDiffEqn}, in which case it is the transport equation.
Another example, in the context of dynamical systems, is when an initial condition $f^0 \in L^2$ is transported by an area-preserving map $M$ via the transfer operator $f^{n+1}= f^n \circ M^{-1}$.

Decay of the correlation function $C_t(g) = \left| \pair{f^t}{g} \right|$  as $t \to \infty$ for observables $g$ in $L^2(\Vol)$ corresponds to mixing of $f^t$ as $t \to \infty$ \cite{SOW}.
\citet{Mathew2005} introduced the $H^{-1/2}$ norm as another criterion to quantify mixing, and \citet{Lin2011b} extended this to any negative Sobolev (e.g., $H^{-q}$) norm and showed that correlations decay to zero if and only if any such ``mix-norm'' decays to zero.
That is,
\begin{equation*}
	\lim_{t \to \infty} \pair{f^t}{g} = 0  \quad \forall g \in L^2 \iff \lim_{t \to \infty} \norm{f^t}_{H^{-q}} = 0 \text{, for any } q>0.
\end{equation*}

Mix-norms are well-suited to quantification of mixing efficiencies \cite{ DoeringThiffeault2006, Lin2010, Shaw2006, Thiffeault2012, Thiffeault2011, Thiffeault2004, Marcotte2018b, Vermach2018}, to lower bounds on the rate of mixing in general \cite{Iyer2014, Lunasin2012, Lunasin2012_erratum}, and to analyzing mixing \cite{ Mathew2003, Mathew2007, Miles2018, Yao2017}.
Moreover, such negative Sobolev spaces provide a natural setting for a discussion of enhanced dissipation and relaxation \cite{Bedrossian_preprint_2,Bedrossian_preprint_3,Constantin2008, CotiZelati2019_preprint, CotiZelati2018, Feng2019, Kiselev2016, Kiselev2008}.
\citet{Mathew2005} introduced the mix-norm in the context of spatial averages over strips, and made the connection to weak convergence (see also \cite{Zillinger2019}).

While mix-norms are well-adapted to the PDE context, correlations and weak convergence are more commonly studied in the context of ergodic theory.
A central question, then, is the quantitative relationship between decay rate of correlations and and decay of mix-norms.
This is the central focus of this paper where we will work in a setting where the evolution of a function $f^t(\xv)$ is given, arising from the continuous-time solution of a PDE or in discrete times from an iterated map.

When studying a collection of functions converging to zero at $t \rightarrow \infty$, such as $\left| \pair{f^t}{g} \right| $ for $g \in X$ where $X \subset L^2$ is some Banach space, there are several reasonable ways to define a rate of decay:
\smallskip
\begin{enumerate}
	\item We can consider a \emph{uniform} upper bound \cite{Bedrossian_preprint_1,Chernov1998, Dolgopyat1998, Liverani1995, Pollicott1981}.
	\medskip
	\item We can say that each function is $\BigOh{\indivBigOh}$ where $\indivBigOh(t)$ is some rate function \cite{Young1999}.
	This lifts the tail of the rate function by multiplying by a constant that depends on $g \in X$.
	\medskip
\item We can instead lift the tail of the rate function by translation and say that each function is bounded above by a translation of some rate function \cite{Elgindi2018}.
\end{enumerate}
\smallskip
We summarize as follows (for concreteness, fix some $q > 0$ and consider $X= H^{q}(\Vol)$):

\begin{enumerate}
	\item Correlations decay at the \emph{uniform} rate $\unif(t)$  for $g \in H^q$ if
	\begin{equation}
			\left| \pair{f^t}{g} \right| \leq \unif(t) \norm{g}_{H^q} ~ \forall g \in H^q \,.
	\end{equation}
	\item Correlations decay at the \emph{asymptotic} rate $\indivBigOh(t)$  for $g \in H^q$ if \begin{equation}\label{dfn:asymptotic}
	\left| \pair{f^t}{g} \right| = \BigOh{\indivBigOh} \text{, for each } g \in H^q \,.
	\end{equation}
	\item Correlations decay at the \emph{translational} rate $\indiv(t)$  for $g \in H^q$ if for each $g \in H^q$ there exists $\tau \in \mathds{R}$ such that for all $t > \tau $ we have \begin{equation}\label{dfn2}
		\left| \pair{f^t}{g} \right| \leq \indiv(t - \tau) \norm{g}_{H^q} \,.
	\end{equation}
\end{enumerate}
\smallskip

Duality implies that the smallest uniform rate is the mix-norm $\norm{f^t}_{H^{-q}} .$
Since any uniform rate trivially satisfies the definitions of asymptotic rate and translational rate, the question is whether there is a $\indivBigOh$ (or $\indiv$) that decays faster than $\norm{f^t}_{H^{-q}}$.
We answer this question by showing that we cannot have  $\indivBigOh = \littleOh{\norm{f^t}_{H^{-q}} }$.
Similarly, given the additional assumption that $\limsup_{t \to \infty}  \indiv(t - \tau) / \indiv(t)$ is finite, we cannot have $\indiv = \littleOh{\norm{f^t}_{H^{-q}} }$.
We do this by constructing an observable $g \in H^q$ such that $\lvert \pair{f^t}{g} \rvert$ decays arbitrarily closely to the mix-norm.
Namely, for any positive $h(t) = \littleOh{\norm{f^t}_{H^{-q}} }$ there is a~$g \in H^q$ such that $\lvert \pair{f^t}{g} \rvert$ is \BigOhText\ but not \littleOhText\ of $h$.

Note that this is not the same as asymptotic equivalence because the correlation may be much smaller than $h$ at certain times.
Moreover, we may take $h = \norm{f^t}_{H^{-q}}$ above if and only if there is a finite set $I$ of Fourier modes where the $H^{-q}$ norm of the projection $P_I f^t$ of $f^t$ onto the modes $I$ is \BigOhText\ but not \littleOhText\ of $\norm{f^t}_{H^{-q}}$.
In this case, we refer to $f^t$ as \emph{\rec} (otherwise it is \emph{\tran}) and the decay of the mix-norm is characterized by $P_I f^t$.

In \cref{Overview} we introduce the key definitions and main theorems.
\cref{Examples} contains examples, and \cref{proofs,proofs2} contains the full proofs to the theorems.

\section{Overview}\label{Overview}

Throughout, it will be more convenient to work with the homogeneous Sobolev spaces $\dot{H}^{\alpha}$ for $\alpha \in \mathds{R}$.
Since the torus $\Vol$ is a compact manifold, Poincar\'e's inequality applies \citep{Hebey} so that the $H^{\alpha}$ norm and $\dot{H}^{\alpha}$ norm are equivalent for mean-zero functions.
For $\alpha>0$, the $\Sobo^{-\alpha}$ norm is typically defined via the duality equation
\begin{equation*}
\norm{\fw}_{\Sobo^{-\alpha}} = \sup\limits_{\substack{ g \in \Sobo^{\alpha}} } \frac{\lvert \pair{f}{g} \rvert }{\norm{g}_{\Sobo^{\alpha}} }.
\end{equation*}
However, there is an equivalent definition \citep{Folland} for all $\alpha \in \mathbb{R}$.
Let $f_{\kv} = \int_{\Vol} f(\xv)\, \ee^{ -2 \pi i  \xv \cdot \kv}~d\xv$ denote the Fourier coefficients of $f(\xv)$.
Then
\begin{equation*}
\norm{f}_{\dot{H}^\alpha} = \left( \sum_{\kv \neq \bm{0}} k^{2\alpha}\, | f_{\kv} |^2 \right)^{\efrac{1}{2}}
\end{equation*}
where  $k^2 = |\kv|^2 = |k_1|^2 + \dots + |k_d|^2$.
We will typically omit the $\kv \neq 0$ on the sum since  $f_{\bm{0}}= 0 $ for mean-zero functions.

Similarly, correlations have a simple expression.
Since the trigonometric functions $\{ \ee^{ 2 \pi i  \xv \cdot \kv} \}$ provide an orthonormal basis for $L^2(\Vol)$, the Fourier transform is a unitary map to $\ell^2(\mathbb{Z}^d)$.
Therefore the Fourier transform preserves the inner product \citep{Folland} and we have Plancherel's Theorem:
\begin{equation*}
\pair{f}{g}  = \sum_{\kv} f_\kv ~\bar{g}_\kv \quad \forall f,g \in L^2(\Vol)  .
\end{equation*}

For time-dependent~$f^t(\xv)$, the duality equation implies $	\left| \pair{f^t}{g} \right|\leq   \norm{\fw^t}_{\Sobo^{-\qq}} \norm{g}_{\Sobo^{\qq}}$ for all $t$.
Moreover, fix $t = t_0$ and take $g$ with Fourier coefficients
\begin{equation}\label{familiar}
g_\kv = f_\kv^{t_0}\,  k^{-2q} \norm{f^{ t_0 }}_{\Sobo^{-q}}^{ - 1}.
\end{equation}
Then $\norm{g}_{\Sobo^{q}} = 1$ and Plancherel's Theorem gives $\pair{f^{t_0}}{g}  = \norm{\fw^{t_0}}_{\Sobo^{-\qq}}$.
The correlation achieves the mix-norm at the time $t_0$.
Since $t_0$ is arbitrary, we see that $ \norm{\fw^t}_{\Sobo^{-\qq} }$ is the envelope of the set of functions $\left| \pair{f^t}{g} \right| $ with $\norm{g}_{\Sobo^{\qq}} = 1$, as in \cref{fig:envelope}.
\begin{figure}
	\begin{tikzpicture}[scale = .6, xscale=1.5,yscale=1.5]
	\draw [<->, ultra thick] (0,4.25) -- (0,-.25) -- (5.25,-.25);
	\draw[black, ultra thick, domain=0:5] plot (\x, {4/(\x+1)  });
	\draw [ultra thick, cyan] (0,1) to [out=45,in=135] (1,2) to [out=-45,in=170] (2,.25) to [out=-10,in=179] (5,.1);
	\draw [fill] (1,2) circle [radius=.05];
	\draw [black] (5,2/3) node [right] {$ \norm{\fw^t}_{\Sobo^{-\qq}}$} ;
	\draw [ultra thick, green] (0,1/2) to [out=0,in=156.037] (2,4/3) to [out=-23.962,in=170] (3,.25) to [out=-10,in=179] (5,.15) ;
	\draw [fill] (2,4/3) circle [radius=.05];
	\draw [ultra thick, red] (0,0) to [out=0,in=165.964] (3,1) to [out=-14.036,in=170] (4,.25) to [out=-10,in=179] (5,.2) ;
	\draw [fill] (3,1) circle [radius=.05];
	\end{tikzpicture}
	\caption{ Plotted above are the mix-norm and correlations with different choices of $g$, demonstrating that the mix-norm is the envelope over $\left| \pair{f^t}{g} \right| $ with $\norm{g}_{\Hq} = 1$.
	}
	\label{fig:envelope}
\end{figure}
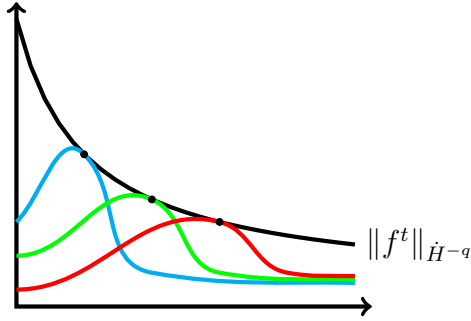

\jlt{My problem with this remark is that we only introduce Fourier energy later.  Maybe keep it briefer?}

\bwo{I agree that it is wordy. How about the following:}

\begin{remark}
	We assume in \cref{familiar} that  $\norm{f^{ t_0 }}_{\Sobo^{-q}} > 0$.
	The case $ \norm{f^{ t_0 }}_{\Sobo^{-q}} = 0$ is degenerate in the context of the advection-diffusion equation and dynamical systems.
	In those settings, if the mix-norm is zero at any finite time it will remain zero for future times.
	Hence we subsequently assume that $ \norm{f^{ t }}_{\Sobo^{-q}} > 0$ for all $t>0$ to simplify the presentation of our results.
\end{remark}

Using only duality, the most that can be said about the relationship between the rate of decay of a correlation and the rate of decay of the mix-norm is that
\begin{equation*}
	\l| \pair{f^t}{g} \r| = \BigOh{\norm{f^t}_{\Sobo^{-\qq}}}  \text{ for each }  g \in \Hq.
\end{equation*}
However, each correlation could decay strictly faster than the mix-norm as illustrated in \cref{fig:envelope}.
We are then led to ask if such a situation is possible.

When is $\l| \pair{f^t}{g} \r| = \littleOh{ \norm{f^t}_{\Sobo^{-\qq}} }$ for each $g \in \Hq\,$?
To answer this question, we must construct functions $g \in \Hq$ such that the correlations $\l| \pair{f^t}{g} \r|$ decay as slowly as possible.
To do this, we first classify $f^t$ as either \rec\ or \tran\ as follows.

For a set $I \subset \mathds{Z}^d$ let
\begin{equation*}
	P_I f^t = \sum_{\kv \in I} f^t_\kv\, \ee^{ 2 \pi i  \xv \cdot \kv}
\end{equation*}
denote the projection of $f^t$ onto the Fourier modes $\kv \in I$.
Then
\begin{equation*}
	\norm{P_I f^{t}}_{\Sobo^{-\qq}}^2 = \sum\limits_{\kv \in I} k^{-2 q}\,|f_\kv^{t}|^2
\end{equation*}
measures the amount of mix-norm supported on $I$.
We often refer to this as the \emph{Fourier energy} contained in~$I$.
This notion of energy is~$q$-dependent, though the~$q$ will usually be clear from the context.

\begin{definition}\label{recurrent}
	We say $f^t$ is \emph{\rec} if there exists a finite set $I \subset \mathds{Z}^d$ such that
	\begin{equation}\label{dfn:qrec:limsup}
	\limsup_{t \to \infty} \frac{\norm{P_I f^t}_{\Sobo^{-\qq}} }{\norm{f^t}_{\Sobo^{-\qq}} } > 0.
	\end{equation}
	Functions that are not \rec\ will be called \emph{\tran}.
\end{definition}

\begin{remark}
We emphasize that $q$-recurrence is a property of $f^t$ which encompasses both the stirring action and the initial condition coupled together.
To clarify, in the context of the advection-diffusion equation~\eqref{advDiffEqn}, $q$-recurrence is a property of a particular realization of $\bm{u}$ and $f^0$ taken together -- it is {\it not} just a property of the vector field $\bm{u}$.
Moreover, a given $f^t$ may be \rec\ for some (larger) values of $q$ and \tran\ for other (smaller) $q$.
\end{remark}

From inequality \eqref{dfn:qrec:limsup} and the trivial bound $\norm{P_I f^t}_{\Sobo^{-\qq}} \leq \norm{f^t}_{\Sobo^{-\qq}}$ we see that $\norm{P_I f^{t}}_{\Sobo^{-\qq}}$ is \BigOhText\ but not \littleOhText\ of $\norm{f^{t}}_{\Sobo^{-\qq}}$.
Unpacking the definition of limit supremum offers another interpretation: there exists $\delta>0$ and a sequence $t_m \to \infty$ where
\begin{equation}\label{support}
\norm{P_I f^{t_m}}_{\Sobo^{-\qq}} \geq \delta  \norm{f^{t_m}}_{\Sobo^{-\qq}}.
\end{equation}
This means that there is at least a $\delta$ fraction of the mix-norm supported on $ I$ at arbitrarily large times.
As time progresses the Fourier energy could move off of $I$, but we can always find a future time $t_{m+1}$ where a proportion $\delta$ of the mix-norm is again on $I$.
In other words, some Fourier energy always returns to populate the spatial scales in $I$.
In this case, test functions $g$ with coefficients for $\kv \in I$ similar to that in \cref{familiar} will match well with $f^t$ at times $t_m$ (after possibly taking a subsequence) so that in \cref{proofs} we can prove the following theorem, a central result of our paper.

\begin{restatable}{theorem}{theoremOne}\label{classify}
	Let $f^t$ be a mean-zero function in $L^2(\Vol)$ with  $\norm{f^{t}}_{\Sobo^{-\qq}} >0$ for all $t >0$. Then $f^t$ is \rec\ if and only if there is a function $g \in \Hq$ such that
	\begin{equation*}
		\limsup_{t \to \infty} \frac{\l| \pair{f^t}{g} \r|}{\norm{f^{t}}_{\Sobo^{-\qq}}} > 0 .
	\end{equation*}
\end{restatable}

Equivalently, there is a function $g \in \Hq$, a constant $c>0$, and a sequence $t_m \to \infty$ where
\begin{equation}
  \l|\pair{f^{t_{m} }}{g}\r| \geq c\, \norm{f^{t_m}}_{\Sobo^{-\qq}}.
  \label{eq:limsupdef}
\end{equation}

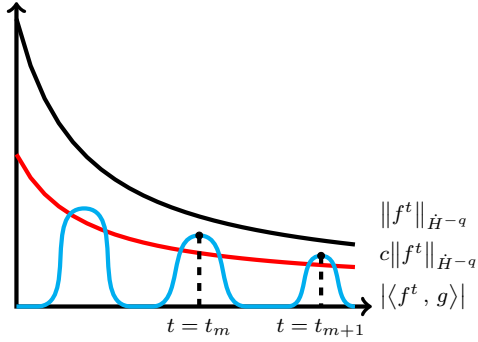
\begin{figure}
	\begin{tikzpicture}[scale = .6, xscale=1.5,yscale=1.5]
	\draw [<->, ultra thick] (0,4.25) -- (0,-.25) -- (5.25,-.25);
	\draw[black, ultra thick, domain=0:5] plot (\x, {4/(\x+1)  });
	\draw[red, ultra thick, domain=0:5] plot (\x, {2/(\x+1)  });
	\draw [ultra thick, cyan] (0, -.25) to [out=0,in=180] (.3 , -.25) to [out=0,in=180] (1,1.2) to [out=0,in=180] (1.7 , -.25) to [out=0,in=180] (2,-.25) to [out=0,in=180] (2.7 , .8) to [out=0,in=180] (3.4 , -.25) to [out=0,in=180] (4, -.25) to [out=0,in=180] (4.5 , .5) to [out=0,in=180] (5,-.25) ;
	\draw [black] (5.2,-.1) node [right] {\tiny $\left| \pair{f^t}{g} \right|$} ;
	\draw [black] (5.2,1.1) node [right] {\tiny $ \norm{\fw^t}_{\Sobo^{-\qq}}$} ;
	\draw [black] (5.2,.5) node [right] {\tiny $ c \norm{\fw^t}_{\Sobo^{-\qq}}$} ;
	\draw [fill] (4.5,.5) circle [radius=.05];
	\draw [fill] (2.7,.8) circle [radius=.05];
	\draw [dashed, ultra thick]  (2.7,.8) -- (2.7,-.25) node [below] {\tiny $t = t_m$};
	\draw [dashed, ultra thick]  (4.5,.5) -- (4.5,-.25) node [below] {\tiny $t = t_{m+1}$};
	\end{tikzpicture}
	\caption{There exists $g \in \Hq$ where $\left| \pair{f^t}{g} \right|$ does not decay faster than $ c \norm{\fw^t}_{\Sobo^{-\qq}}$.}
	\label{fig:digest}
\end{figure}

\begin{remark}
As demonstrated in \cref{fig:digest}, it is possible that $\l|\pair{f^{t }}{g}\r|$ is small at times $t \neq t_m$ and so we do not show asymptotic equivalence.
We interpret our result as demonstrating that the correlation does not decay asymptotically faster than the mix-norm in the sense that $\l| \pair{f^t}{g} \r|$ is \BigOhText\ but not \littleOhText\ of $\norm{\fw^t}_{\Sobo^{-\qq}}$.
From this theorem, the answer to our previously posed question `when is $\l| \pair{f^t}{g} \r| = \littleOh{ \norm{f^t}_{\Sobo^{-\qq}} }$ for each $g \in \Hq\,$?' is exactly when $f^t$ is \tran.
\end{remark}

This naturally prompts us to ask if $f^t$ is \tran\ and we carefully choose $g \in \Sobo^q$, how slowly can we make $ \l| \pair{f^t}{g} \r|$ decay?
The following theorem answers this question.
\begin{restatable}{theorem}{maintheorem}\label{maintheoremLabel}
	Let $f^t$ be a mean-zero function in $L^2(\Vol)$ with  $\norm{f^{t}}_{\Sobo^{-\qq}} >0$ for all $t >0$. For any positive function $h(t)$ such that $\rate(t) = \littleOh{ \mixnorm{f^t} }$, there is a function $g \in \Hq$ such that
	\begin{equation*}
	\limsup_{t \to \infty} \frac{\l| \pair{f^t}{g} \r|}{\rate(t)} > 0 .
      \end{equation*}
\end{restatable}

\begin{remark}
\cref{classify,maintheoremLabel} do not require $f^t$ to be bounded uniformly in $L^2(\Vol)$ in time, nor for the mix-norm to decay to zero.
Additionally, \cref{maintheoremLabel} is valid whether $f^t$ is \rec\ or \tran.
\end{remark}

The proof of \cref{maintheoremLabel} is deferred until \cref{proofs2}, but we present the idea behind the proof now.
If $f^t$ is \rec, then the proof is accomplished by a result similar to \cref{classify}.
For \tran\ functions, the proof relies on the construction of sets $I_m$ and times $t_m$ satisfying certain properties, the first being that we want the finite disjoint sets $I_m \in \mathds{Z}^\sdim$ to capture a large amount of the Fourier energy at time $t_m$.
We can do this since $q$-transience ensures that we can wait for the next time $t_{m+1}$ for a proportion of the Fourier energy for moves off of $I_m$ and never comes back.
Then by choosing the Fourier coefficients of $g$ on $I_m$ to agree with $f^t$ at time $t_m$, we can guarantee that $\l| \pair{f^t}{g} \r|$ will be large at time $t_m$. Hence, the function $g$ in \cref{maintheoremLabel} which gives the slowly decaying $\l| \pair{f^t}{g} \r| $ has Fourier coefficients
\begin{equation}\label{familiar2}
g_\kv =
\begin{cases}
f_\kv^{t_m}  k^{-2q} \norm{f^{t_m}}_{\Sobo^{-q}}^{ - 2} \,\rate(t_m), & \kv \in I_m; \\
0, & \text{otherwise}; \\
\end{cases}
\end{equation}
where $I_m$ are disjoint and $\norm{P_{I_m}f^{t_m}}_{\Sobo^{-q}}$ captures a nonzero proportion of $h(t_m) $, similar to inequality~\eqref{support}. These coefficients are similar to those of \cref{familiar} except with an extra factor of $h/\mixnorm{f}$. This factor is needed so that we can satisfy the second property we require from the sets $I_m$ and times $t_m$: by taking a subsequence, we can use the fact that $\rate(t) = \littleOh{ \mixnorm{f^t} }$ to make the $g_\kv$ decay fast enough as $k \to \infty$ to have $g \in \Sobo^q$.
Hence, although correlations may not achieve the decay rate of the mix-norm, they may achieve the decay rate of $h$.

These theorems allow us to show the result outlined in the introduction.
The following corollary reveals it is not possible to find a $\indivBigOh$ or $\indiv$ (under a given growth condition) that is \littleOhText\ of the mix-norm.
We remark that the growth condition on $\indiv$ that $\limsup_{t \to \infty}  \indiv(t - \tau) / \indiv(t)$ is finite is satisfied by power law and exponential functions but not by double exponential functions.
\begin{corollary}\label{indivVersusMixnorm}
	\
	\begin{enumerate}
		\item 	For any $\indivBigOh$ satisfying \cref{dfn:asymptotic}, we have
		\begin{equation*}
		\limsup_{t \to \infty} \frac{\indivBigOh(t)}{\norm{f^t}_{H^{-q}}} > 0 \,.
		\end{equation*}

		\item For $\indiv$ satisfying \cref{dfn2} and $\limsup_{t \to \infty}  \indiv(t - \tau) / \indiv(t)$ finite for any $\tau \in \mathds{R}$, we have
		\begin{equation*}
		\limsup_{t \to \infty} \frac{\indiv(t)}{\norm{f^t}_{H^{-q}}} > 0 \,.
		\end{equation*}
	\end{enumerate}

\end{corollary}

\begin{proof}[Proof of \cref{indivVersusMixnorm}]
	Seeking contradiction we suppose there is a~$\indivBigOh(t)$ satisfying \cref{dfn:asymptotic} such that $\indivBigOh(t) = \littleOh{\norm{f^t}_{\Sobo^{-\qq}}}$.
	Choosing $h(t) = \sqrt{\indivBigOh(t) \norm{f^t}_{\Sobo^{-\qq}}}$, the geometric mean of $\indivBigOh$ and the mix-norm, we see
	\begin{equation}\label{indivVersusMixnorm:helper}
		\limsup_{t \to \infty} \frac{h}{\norm{f^t}_{\Sobo^{-\qq}} } = \limsup_{t \to \infty} \sqrt{ \frac{\indivBigOh}{\norm{f^t}_{\Sobo^{-\qq}}} } = 0\,.
	\end{equation}
	Then \cref{maintheoremLabel} assures there is a~$g \in \Hq$ with
	\begin{equation*}
		\limsup_{t \to \infty} \frac{ \l| \pair{f^t}{g} \r| }{h} > 0\,.
	\end{equation*}
	Then we have arrived at a contradiction:
	\begin{equation*}
		\limsup_{t \to \infty} \frac{ \l| \pair{f^t}{g}\r| }{h} \leq  \limsup_{t \to \infty} \frac{ \l| \pair{f^t}{g}\r| }{\indivBigOh} \cdot \limsup_{t \to \infty} \frac{ \indivBigOh}{h} = 0
	\end{equation*}
	since
	$	\limsup_{t \to \infty}  \l| \pair{f^t}{g} \r| / \indivBigOh 	$ is finite by \cref{dfn:asymptotic} and
	$	\limsup_{t \to \infty}  \indivBigOh / h = 0 $ as in \cref{indivVersusMixnorm:helper}.

	 A similar argument gives us the second half of the corollary.
	 In this case choose $h(t) = \sqrt{\indiv(t) \norm{f^t}_{\Sobo^{-\qq}}}$ and apply \cref{maintheoremLabel} to produce a test function $g$ which comes with a $\tau$ from \cref{dfn2}.
	 Then we have another contradiction:
	 \begin{equation*}
	 \limsup_{t \to \infty} \frac{ \l| \pair{f^t}{g}\r| }{h(t)} \leq  \limsup_{t \to \infty} \frac{ \l| \pair{f^t}{g}\r| }{\indiv(t - \tau)} \cdot \limsup_{t \to \infty} \frac{ \indiv(t - \tau) }{\indiv(t)} \cdot \limsup_{t \to \infty}  \frac{ \indiv(t)}{h(t)} = 0
	 \end{equation*}
	 since $\limsup_{t \to \infty} \indiv(t - \tau) / \indiv(t) $ is finite by hypothesis.
\end{proof}

In \cref{Examples} we present an example of a \tran\ function and alter it to send energy down the spectrum \textit{less} efficiently, resulting in $q$-recurrence for a range of $q$.
We then include diffusion at every time step, demonstrating a transition to $q$-recurrence for all $q>0$.
We include a numerical example and provide intuition about how to recognize when~$f^t$ is \rec.
Finally, we prove the theorems in generality in \cref{proofs,proofs2}.

\section{Examples}\label{Examples}

\begin{Example}[baker's map and $q$-transience]
\begin{figure}
	\begin{tikzpicture}[xscale=1.5,yscale=1.5]
	\draw (0, 0) rectangle (0.5, 1);
	\draw [fill=blue] (1,0) rectangle (0.5,0.5);
	\draw [fill=red] (1,1) rectangle (0.5,0.5);
	\draw [->, ultra thick] (1.1,0.5) -- (2,0.5) ;
	\end{tikzpicture}
	\begin{tikzpicture}[xscale=1.5,yscale=1.5]
	\draw (0, 0) rectangle (0.25, 2);
	\draw [fill=blue] (.5,0) rectangle (0.25,1);
	\draw [fill=red] (.25,1) rectangle (0.5,2);
	\end{tikzpicture}
	\begin{tikzpicture}[xscale=1.5,yscale=1.5]
	\draw (0, 0) rectangle (0.25, 1);
	\draw [fill=blue] (.5,0) rectangle (0.25,1);
	\draw (0.75, 0) rectangle (0.25, 1);
	\draw [fill=red] (1,0) rectangle (.75,1);
	\draw [->, ultra thick] (-1,1.6) to [out=0,in=90] (.75,1.05);
	\end{tikzpicture}
	\caption{The baker's map.}
	\label{fig:baker}
\end{figure}
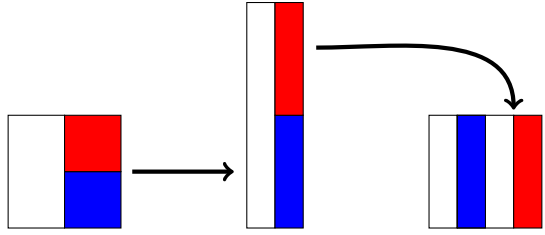
Let $B$ be the baker's map, the area preserving transformation of the domain $[ 0, 1 ]^2$ as pictured in \cref{fig:baker}.
For the $y$-independent initial function $f^0(x,y) = 2\cos \left( 2 \pi x \right)$, applying the baker's map simply doubles the frequency in the $x$ direction.
After $n$ applications of the baker's map we have $f^n = f^0 \circ B^{-n}$ $ = 2\cos \left(  2 \pi~2^n x \right)$.
As a result, the Fourier coefficients have the simple expression
\begin{equation*}
f_{\kv}^n = \begin{cases}
1 & k_1 = 2^n, k_2 = 0; \\
0 & \text{otherwise}. \label{coefficients}
\end{cases}
\end{equation*}
This is a one dimensional action on Fourier coefficients $f^n_k = f^n_{k_1, 0}$ via an infinite dimensional matrix $A_{k\ell}$ as
\begin{equation}
f^{n+1}_k = \sum_{\ell} A_{k\ell}\, f^{n}_\ell
\label{eq:Af}
\end{equation}
where
\begin{equation*}
\begin{tikzpicture}[baseline=(current  bounding  box.center)]
\matrix (m) [matrix of math nodes,
nodes={rectangle, 
	minimum size=1.2em, text depth=0.25ex,
	inner sep=0pt, outer sep=0pt,
	fill opacity=0.5, text opacity=1,
	anchor=center},
column sep=-0.5\pgflinewidth,
row sep=-0.5\pgflinewidth,
row 2/.append style = {nodes={fill=gray!50}},
row 4/.append style = {nodes={fill=gray!50}},
row 6/.append style = {nodes={fill=gray!50}},
row 8/.append style = {nodes={fill=gray!50}},
inner sep=0pt,
left delimiter=(, right delimiter=),
]
{
	\hspace{5mm} & \hspace{5mm} & \hspace{5mm} & \hspace{5mm} & \hspace{5mm} \\
	1 & \hspace{5mm} & \hspace{5mm} & \hspace{5mm} & \hspace{5mm} \\
	\hspace{5mm} & \hspace{5mm} & \hspace{5mm} & \hspace{5mm} & \hspace{5mm} \\
	\hspace{5mm} & \hspace{1mm} 1 \hspace{1mm} & \hspace{5mm} & \hspace{5mm} & \hspace{5mm} \\
	\hspace{5mm} & \hspace{5mm} & \hspace{5mm} & \hspace{5mm} & \hspace{5mm} \\
	\hspace{5mm} & \hspace{5mm} & \hspace{1mm} 1 \hspace{1mm} & \hspace{5mm} & \hspace{5mm} \\
	\hspace{5mm} & \hspace{5mm} & \hspace{5mm} & \hspace{5mm} & \hspace{5mm} \\
	\hspace{5mm} & \hspace{5mm} & \hspace{5mm} & \hspace{1mm} 1 \hspace{1mm} & \hspace{5mm} \\
	\hspace{5mm} & \hspace{5mm} & \hspace{5mm} & \hspace{5mm} &  \ddots  \\
};
\foreach \i [count=\xi from 1] in  {1,...,4}{
	\node also [label=above:{\footnotesize \xi}] (m-1-\i) {};
	\node also [label= { [label distance=3mm] right: {\footnotesize \xi} }] (m-\i-5) {};
}
\node also [label=above:{\footnotesize \dots}] (m-1-5) {};
\foreach \i [count=\xi from 5] in  {5,...,8}{
	\node also [label= { [label distance=3mm] right: {\footnotesize \xi} }] (m-\i-5) {};
}
\node also [label= { [label distance=3.5mm] right: {\footnotesize \vdots} }] (m-9-5) {};
\node also [label= { [label distance=5mm] left: { $(A_{k\ell})=$ } }] (m-5-1) {};
\end{tikzpicture}
\label{eq:Bmatrix}
\end{equation*}
is populated by 1's along a subdiagonal of slope $-2$ and 0's everywhere else.

The entire mix-norm is supported on just one Fourier mode and given any finite set $I \in \mathds{Z}^d$, it is clear that, as $n$ increases, the Fourier energy will move off of $I$ and never return.
Therefore $f^n$ is  \tran\ $\forall q>0$.
\end{Example}

\begin{Example}[altered baker's map action and $q$-recurrence]\label{ExAltBak}
We now alter the previous example so that the energy of $f^n$ is sent down the spectrum \emph{less} effectively, the result being a \emph{$q$-recurrent} function (if $q$ is large enough).
This time, consider the action on the Fourier coefficients of $f^n(x)$ as in \cref{eq:Af} via the infinite dimensional matrix
\begin{equation*}
	\begin{tikzpicture}[baseline=(current  bounding  box.center)]
		\matrix (m) [matrix of math nodes,
		nodes={rectangle, 
		minimum size=1.2em, text depth=0.25ex,
		inner sep=0pt, outer sep=0pt,
		fill opacity=0.5, text opacity=1,
		anchor=center},
		column sep=-0.5\pgflinewidth,
		row sep=-0.5\pgflinewidth,
		row 2/.append style = {nodes={fill=gray!50}},
		row 4/.append style = {nodes={fill=gray!50}},
		row 6/.append style = {nodes={fill=gray!50}},
		row 8/.append style = {nodes={fill=gray!50}},
		inner sep=0pt,
		left delimiter=(, right delimiter=),
		]
		{
		\aaa & \hspace{5mm} & \hspace{5mm} & \hspace{5mm} & \hspace{5mm} \\
		\bbb & \hspace{5mm} & \hspace{5mm} & \hspace{5mm} & \hspace{5mm} \\
		\hspace{5mm} & \hspace{5mm} & \hspace{5mm} & \hspace{5mm} & \hspace{5mm} \\
		\hspace{5mm} & \hspace{1mm} 1 \hspace{1mm} & \hspace{5mm} & \hspace{5mm} & \hspace{5mm} \\
		\hspace{5mm} & \hspace{5mm} & \hspace{5mm} & \hspace{5mm} & \hspace{5mm} \\
		\hspace{5mm} & \hspace{5mm} & \hspace{1mm} 1 \hspace{1mm} & \hspace{5mm} & \hspace{5mm} \\
		\hspace{5mm} & \hspace{5mm} & \hspace{5mm} & \hspace{5mm} & \hspace{5mm} \\
		\hspace{5mm} & \hspace{5mm} & \hspace{5mm} & \hspace{1mm} 1 \hspace{1mm} & \hspace{5mm} \\
		\hspace{5mm} & \hspace{5mm} & \hspace{5mm} & \hspace{5mm} &  \ddots  \\
		};
		\foreach \i [count=\xi from 1] in  {1,...,4}{
		\node also [label=above:{\footnotesize \xi}] (m-1-\i) {};
		\node also [label= { [label distance=3mm] right: {\footnotesize \xi} }] (m-\i-5) {};
		}
		\node also [label=above:{\footnotesize \dots}] (m-1-5) {};
		\foreach \i [count=\xi from 5] in  {5,...,8}{
		\node also [label= { [label distance=3mm] right: {\footnotesize \xi} }] (m-\i-5) {};
		}
		\node also [label= { [label distance=3.5mm] right: {\footnotesize \vdots} }] (m-9-5) {};
		\node also [label= { [label distance=5mm] left: { $(\widetilde{A}_{k\ell})=$ } }] (m-5-1) {};
	\end{tikzpicture}
        \label{eq:Amatrix}
\end{equation*}
where $\aaa, \bbb>0$ are constants such that $\aaa^2 + \bbb^2 = 1$.

\begin{remark}
	It is not evident that the current example is still a dynamical systems example.
	That is, we do not know that there is a map $T: [0,1]^d \to [0,1]^d$ so that
	$f^{n+1} = f^n \circ T$
	and
	$f^{n+1}_k = \sum_{\ell} \widetilde{A}_{k\ell}\, f^{n}_\ell$~.
Moreover, such a map might not be injective, surjective, or unique.
For example, $B^{-1}$ from the previous example is not injective when thought of as a map on $[0,1]$, but it is when we allow it to move through another dimension ($d=2$).
For the present example, taking $\aaa, \bbb = \sqrt{2}/2$, the initial function $f^0(x) = 2\cos \left( 2 \pi x \right)$ is transformed to $f^1(x) = \sqrt{2} ( \aaa \cos \left( 2 \pi x \right) +  \bbb \cos \left( 4 \pi x \right)  )$ after one time step.
Since the range of $f^0$ and $f^1$ are not the same set, any such $T$ cannot be surjective.
Lastly, if $f^n$ is constant, then $T$ can be any map.
These caveats notwithstanding, we consider the current example as a study of possible ways for energy to move down the spectrum and proceed with an analysis.
\end{remark}

We show the coefficients of $f^n$ for initial function $f^0(x) = 2\cos \left( 2 \pi x \right)$ in \cref{coeff_baker}. Heuristically, the energy starts concentrated on the $k=1$ mode and subsequently splits between modes $k=1,2$ so that $L^2$ norm is preserved.
After that the $k=1$ mode continues to donate a proportion $\bbb$  of its energy to $k=2$ and the energy on $k=2$ is transported down the spectrum at the same rate as the baker's map ($k=2^n$).

Notice that $f^n$ is mean-zero because it is the finite sum of cosine functions with full period.
We can compute the $L^2$ norm and find $\norm{f^n}_{L^2} = 1 ~\forall n$.
Hence $A_{k\ell}$ is a unitary map on $\ell^2$ by the polarization identity.
Therefore $f^n$ is bounded uniformly in $L^2$ and so $f^n$ satisfies the hypothesis of the theorems in \cref{Overview}.

\begin{table*} \small
	\ra{1.3}
	\setlength\arrayrulewidth{1.5pt}
		\begin{tabular}{@{}l|p{12mm}p{12mm}p{12mm}p{12mm}p{12mm}p{12mm}p{12mm}p{12mm}p{12mm}p{12mm}  @{}}
			& \footnotesize{$k=1$} & 2 & 3 & 4 & 5 & 6 & 7 & 8 & \dots \\ \hline

			$f^0_k$ & 1 &&&&&&&& \\[3pt]
			\rowcolor[gray]{.9}
			$f^1_k$ & $\aaa$ & $\bbb$ &&&&&&&\\[3pt]
			$f^2_k$ & $\aaa^2$ & $\aaa \bbb$ & & $\bbb$ &&&&&\\[3pt]
			\rowcolor[gray]{.9}
			$f^3_k$ & $\aaa^3$ & $\aaa^2 \bbb$ & & $\aaa \bbb$ & & & & $\bbb$ &\\[3pt]
			\vdots &&&&&&&&
		\end{tabular}%
        \medskip
	\caption{Nonvanishing Fourier coefficients of $f^n$ defined by \cref{eq:Af}, for $f^0(x) = 2\cos \left( {2 \pi} x \right)$.}
	\label{coeff_baker}
\end{table*}

Consider the contribution to the mix-norm from mode $k=1$
\begin{equation*}
	\mathcal{E}_1^n = \norm{P_{k = 1} f^{n}}_{\Sobo^{-\qq}}^2 = \sum\limits_{k = 1} k^{-2 q}|f_k^{n}|^2  = |f_1^{n}|^2 = \aaa^{2n},
\end{equation*}
and compare it to the contributions from modes $k>1$ (a geometric sum):
\begin{align*}
	\mathcal{E}_{k>1}^n &= \norm{P_{k>1} f^{n}}_{\Sobo^{-\qq}}^2 = \sum\limits_{k > 1} k^{-2 q}|f_k^{n}|^2\\
	& =
	\begin{cases}
		c_{q,\aaa} \left(  \aaa^{2n}  - 2^{-2qn} \right), & \text{for } \aaa \neq 2^{-q}; \vspace{2mm}\\
		\bbb^2 ~\aaa^{2n}  n, & \text{for } \aaa = 2^{-q};
	\end{cases}
\end{align*}
where
\begin{equation*}
c_{q,\aaa} = \frac{ b^2} {\aaa^2 2^{2q} -  1 }.
\end{equation*}
For $\aaa >  2^{-q}$, we see that $\mathcal{E}_1^n \sim \mathcal{E}_{k>1}^n$ as $n \to \infty$ and the mode $k = 1$ captures a non-zero proportion of the mix-norm for arbitrarily large $n$.
Therefore $f^n$ is \rec\ for $q > log_2(1/\aaa)$.

For $\aaa \leq  2^{-q}$, $\mathcal{E}_1^n = o \left( \mathcal{E}_{k>1}^n \right)$ so the mode $k = 1$ does not capture a non-zero proportion of the mix-norm for arbitrarily large $n$.
This suggests that $f^n$ is \tran\ for $q \leq log_2(1/\aaa)$.
To prove $q$-transience, we need to show the same holds for an arbitrary finite set $I$.
Take $I = \left[ 0,2^R \right]$ for some $R \in \mathds{N}$.
For $n> R$,
\begin{equation*}
	\norm{P_I f^n}_{\Sobo^{-q}}^2 = c_{ R, q, \aaa}  ~ \aaa^{2n}
\end{equation*}
where
\begin{equation*}
	c_{ R, q, \aaa} =
	\begin{cases}
		1 +  c_{q,\aaa}  \left( 1 - \aaa^{-2R} 2^{-2qR} \right), & \text{for } \aaa \neq 2^{-q}; 	\vspace{2mm}\\
		1 + \bbb^2  R, & \text{for } \aaa = 2^{-q}
\end{cases}
\end{equation*}
and we conclude that
\begin{equation*}
	\lim_{n \to \infty} \frac{\norm{P_I f^n}^2_{\Sobo^{-q}} }{\norm{ f^n}^2_{\Sobo^{-q}}} =
	\begin{cases}
	c_{R, q, \aaa} \left( 1 + c_{q,\aaa} \right)^{-1}  , & \text{for } \aaa > 2^{-q};\\
	0, & \text{for } \aaa \leq 2^{-q}.
	\end{cases}
\end{equation*}
Therefore $f^n$ is \tran\ for $q \leq log_2(1/\aaa)$ and \rec\ for $q > log_2(1/\aaa)$.
\end{Example}

\begin{Example}[altered baker's map action with diffusion]\label{ExDiff}
We use the same matrix $\widetilde{A}_{k\ell}$ and initial condition as in \cref{ExAltBak} but now we also include diffusion.
Without diffusion the conclusion from the previous section was that $f^n$ is \rec\ if $q$ was large enough.
We now show that with diffusion, $f^n$ is \rec\ for all $q$.
Along the way, we show the rate of decay of the Sobolev norm $\norm{f^n}_{H^\beta}$ is independent of $\beta \in \mathds{R}$.

Let
\begin{equation}
	\gamma_k = \exp\l(- \kappa (2 \pi k)^2\r)
\end{equation}
and update the Fourier coefficients according to
\begin{equation}
	f^{n+1}_k = \sum_{\ell} \gamma_k  \widetilde{A}_{k\ell}\, f^{n}_\ell
        \label{eq:DAf}
\end{equation}
where~$\widetilde{A}_{k\ell}$ is the matrix defined in~\cref{eq:Amatrix}.
This matrix multiplication will result in \emph{pulsed diffusion} with diffusion constant $\kappa$.
We display the coefficients of $f^n$ in \cref{coeff_diff}.

\begin{table*} \small
	\ra{1.3}
        \setlength\arrayrulewidth{1.5pt}
	\begin{tabular}{@{}l|p{12mm}p{12mm}p{12mm}p{14mm}p{12mm}p{12mm}p{12mm}p{12mm}p{12mm} @{}}
		& \footnotesize{$k=1$} & 2 & 3 & 4 & 5 & 6 & 7 & 8 & \dots \\ \hline

		$f^0_k$ & 1 &&&&&&&& \\[3pt]
		\rowcolor[gray]{.9}
		$f^1_k$ & $\aaa \gamma_1$ & $\bbb \gamma_2$ &&&&&&&\\[3pt]
		$f^2_k$ & $\aaa^2 \gamma_1^2$ & $\aaa \gamma_1 \bbb \gamma_2$ & & $\bbb \gamma_2 \gamma_4$ &&&&&\\[3pt]
		\rowcolor[gray]{.9}
		$f^3_k$ & $\aaa^3 \gamma_1^3$ & $\aaa^2 \gamma_1^2 \bbb \gamma_2$ & & $\aaa \gamma_1 \bbb \gamma_2 \gamma_4$ & & & & $\bbb \gamma_2 \gamma_4 \gamma_8$ &\\[3pt]
		\vdots &&&&&&&&
	\end{tabular}%
        \medskip
	\caption{Nonvanishing Fourier coefficients of $f^n$ defined by \cref{eq:DAf}, for $f^0(x) = 2\cos \left( {2 \pi} x \right)$.}
	\label{coeff_diff}
\end{table*}

The amount of mix-norm found on mode $k=1$,
\begin{equation}
	\mathcal{E}_1^n = \norm{P_{k = 1} f^{n}}_{\Sobo^{\beta}}^2 = \sum\limits_{k = 1} k^{2 \beta}|f_k^{n}|^2  = |f_1^{n}|^2 = (\aaa \gamma_1)^{2n},
\end{equation}
is asymptotically equivalent to the amount found on modes $k>1$ because
\begin{align*}
	\mathcal{E}_{k>1}^n &= \norm{P_{k>1} f^{n}}_{\Sobo^{\beta}}^2 = \sum\limits_{k > 1} k^{2 \beta}|f_k^{n}|^2\\
	& =  \sum_{\ell = 1}^n (2^\ell)^{2 \beta} \left( (\aaa \gamma_1)^{n - \ell} b ~ \Pi_{s=1}^\ell \gamma_{2^s} \right)^2\\
	& = (\aaa \gamma_1)^{ 2n } \sum_{\ell = 1}^n (2^\ell)^{2 \beta} \left( (\aaa \gamma_1)^{ - \ell} b ~ \Pi_{s=1}^\ell \gamma_{2^s} \right)^2
\end{align*}
where the factor
\begin{equation}
	c_{n, \beta, \aaa, \kappa} := \sum_{\ell = 1}^n (2^\ell)^{2 \beta} \left( (\aaa \gamma_1)^{ - \ell} b ~ \Pi_{s=1}^\ell \gamma_{2^s} \right)^2
\end{equation}
limits to a finite constant $c_{ \beta, \aaa, \kappa}$ as $n \to \infty$.
We see that $c_{n, \beta, \aaa, \kappa}$ converges since the factor
\begin{equation}
	\gamma_{2^\ell} = \exp\l(-  \kappa (2\pi 2^\ell)^2    \r)
\end{equation}
dominates the terms in the sum to render the sum convergent.
Lastly, notice that $c_{ \beta, \aaa, \kappa} \to \infty$ as $\beta \to \infty$ or $\aaa \to 0$.
We conclude that $f^n$ is \rec\ for all $q \in \mathds{R}$ and, moreover, that all of the Sobolev norms decay at the same rate.
\end{Example}

\bwo{JL To Do: Please include a one line summary of what the sine flow is. Also, Charlie asked if the `small amount of diffusion' is actual diffusion with constant $\kappa$ or numerical diffusion?
}
\begin{Example}[sine flow]
 Lastly we consider a computational example, the random sine flow, which is a simple model flow that is empirically quite efficient at mixing \citep{Pierrehumbert1994,Thiffeault2004}.
 The sine flow is a two-dimensional time-periodic flow with a full period consisting of the shear flow
 \begin{subequations}
   \begin{align}
   \bm{u}_1(t,x) = \sqrt{2}\,(0\,,\,\sin(2\pi x + \psi_1)),
  \qquad
  0 \le t < 1/2,
\intertext{followed by}
  \bm{u}_2(t,y) = \sqrt{2}\,(\sin(2\pi y + \psi_2)\,,\,0),
  \qquad
  1/2 \le t < 1,
\end{align}
\label{eq:sineflow}%
\end{subequations}
with~$(x,y) \in [0,1]^2$ and periodic spatial boundary conditions.
Here~$\psi_1$ and~$\psi_2$ are random phases, uniformly distributed in~$[0,2\pi]$, chosen independently at every period.
Unlike the pulsed diffusion in \cref{ExDiff}, diffusion acts continuously by solving the advection--diffusion equation~\eqref{advDiffEqn} with diffusivity~$D = 10^{-5}$.
We display $\norm{f^t}_{H^{{-q}}}$ for various $q$ in \cref{fig:numerical}, for initial condition~$f^0(x) = \sqrt{2}\cos(2\pi x)$, and observe that the mix-norms all decay at the same rate, at least within numerical fluctuations.
\begin{figure}
	\begin{center}
		\includegraphics[width=.7\textwidth]{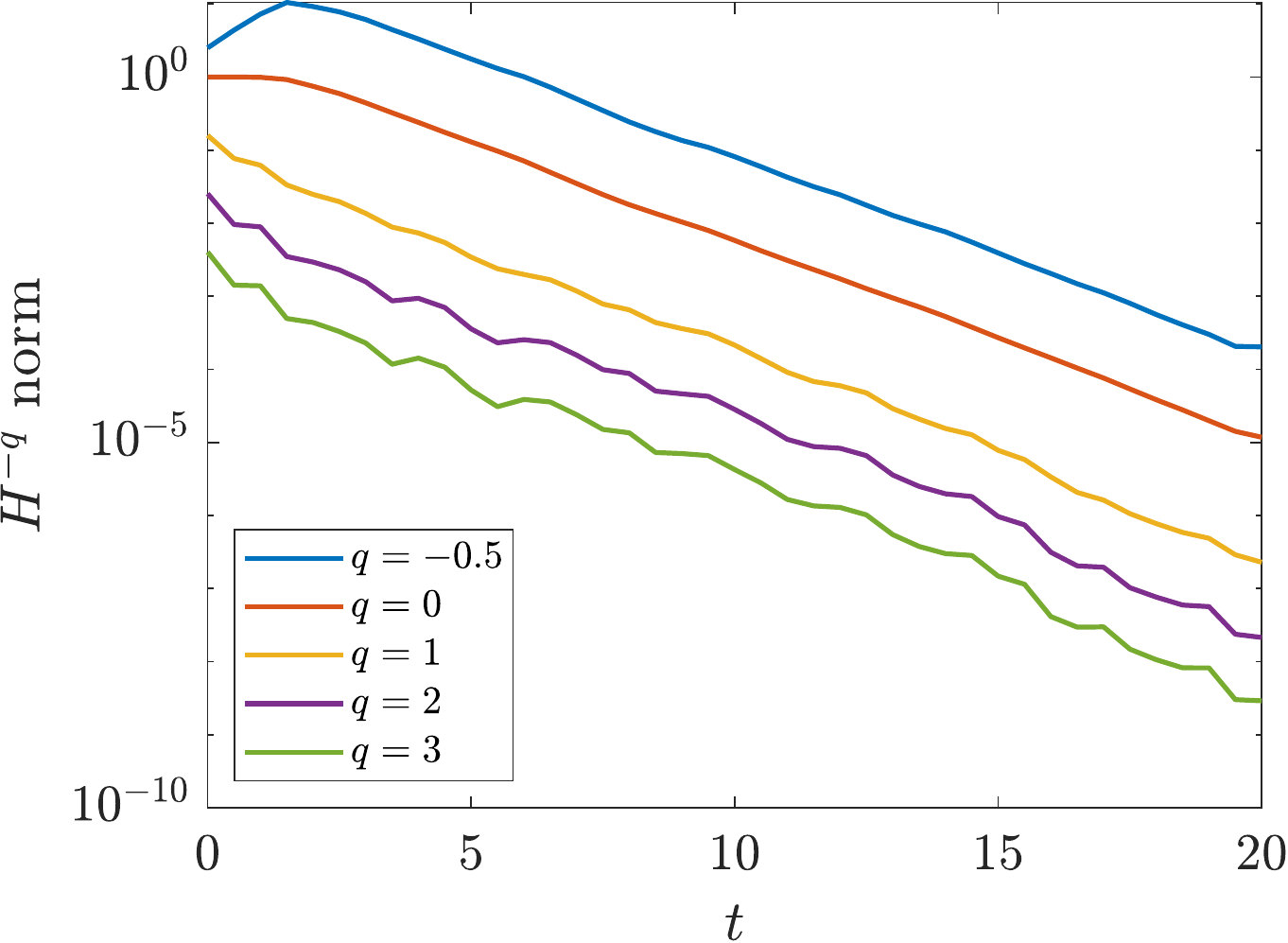}
	\end{center}
	\caption{For the advection-diffusion equation~\eqref{advDiffEqn} with~$\bm{u}$ given by the random sine flow~\eqref{eq:sineflow}, the rate of decay of the mix-norms is independent of $q$.  The the initial condition is~$f^0(x) = \sqrt{2}\cos(2\pi x)$, and the diffusivity is~$D = 10^{-5}$.}
	\label{fig:numerical}
\end{figure}
\end{Example}

In general, if $f^t$ is \rec\ then the decay rate of the mix-norm is independent of $q$ in the following sense:
\begin{theorem}
	If $f^t$ is \rec, then it is also $q'$-recurrent for any $q' > q$. Moreover, we have
	\begin{equation*}
		\limsup_{t \to \infty} \frac{\norm{f^t}_{\Sobo^{-q'}} }{\norm{f^t}_{\Sobo^{-q}}} > 0\,.
	\end{equation*}
	Then together with the trivial estimate
	\begin{equation}\label{independent:helper}
		\norm{f^t}_{\Sobo^{-q'}} \leq \norm{f^t}_{\Sobo^{-q}}
	\end{equation}
	we conclude that $\norm{f^t}_{\Sobo^{-q'}}$ is \BigOhText\ but not \littleOhText\ of $\norm{f^t}_{\Sobo^{-q}}$.
\end{theorem}

\begin{proof}
	Since $f^t$ is \rec, there is a finite set $I$ such that
	\begin{equation}\label{independent:helper1}
		0 < \limsup_{t \to \infty} \frac{\norm{P_I f^t}_{\Sobo^{-q}} }{\norm{ f^t}_{\Sobo^{-q}}}\,.
	\end{equation}
	Say $I \subset [ -R, R]$ for some $R \in \mathds{N}$. Then
	\begin{equation}\label{independent:helper2}
		\norm{P_I f^t}_{\Sobo^{-q}} \leq R^{(q'-q)} \norm{P_I f^t}_{\Sobo^{-q'}}\,.
	\end{equation}
	Putting together \cref{independent:helper,independent:helper1,independent:helper2} we obtain
	\begin{equation*}
		0 < R^{(q'-q)} \limsup \frac{\norm{P_I f^t}_{\Sobo^{-q'}} }{\norm{ f^t}_{\Sobo^{-q'}}}.
	\end{equation*}
	We conclude $f^t$ is $q'$-recurrent. Moreover, the trivial estimate $\norm{P_I f^t}_{\Sobo^{-q'}} \leq \norm{ f^t}_{\Sobo^{-q'}}$ together with \cref{independent:helper1,independent:helper2} imply

	\begin{equation}
		0 < R^{(q'-q)} \limsup \frac{\norm{ f^t}_{\Sobo^{-q'}} }{\norm{ f^t}_{\Sobo^{-q}}}.
	\end{equation}

\end{proof}

One question for further investigation is whether a converse to the above theorem exists.
That is, can we conclude $f^t$ is \rec\ for a range of $q$ if the mix-norms decay at the same rate for the those $q$?
Another question concerns the transition from \tran\ to \rec\ when including pulsed diffusion.
Does $q$-recurrence imply an introduction of the Batchelor scale and anomalous dissipation \cite{Bedrossian_preprint_4,Miles2018b,Miles2018}?

\section{Proof of Theorem 1}\label{proofs}

We begin by generalizing the definition of \rec\ functions to the notion of `\typeone\~' functions.

\begin{definition}\label{def:qhrec}
	For positive functions $h(t)$, we say $f^t$ is \typeone\ if there exists a finite set $I \subset \mathds{Z}^d$ such that
	\begin{equation}
	\limsup_{t \to \infty} \frac{\norm{P_I f^t}_{\Sobo^{-\qq}} }{h} > 0.
	\end{equation}
	Functions that are not \typeone\ are called \typetwo.
\end{definition}

\begin{lemma}
	\label{case1}
	If $f^t$ is \typeone , then there is a function $g \in \Hq$ such that
	\begin{equation*}
	\limsup_{t \to \infty} \frac{\l| \pair{f^t}{g} \r|}{h} > 0 .
	\end{equation*}
	Moreover, $g \in \Sobo^{\beta}$ for any $\beta \in \mathds{R}$ with
	\begin{equation}
		\norm{g}_{\Sobo^{\beta} }^2 = 2 \sum\limits_{\kv \in I} k^{ 2(\beta - q) }  .
	\end{equation}
\end{lemma}

\begin{proof}
There exists a constant $\constant > 0$ and a sequence of times $t_m \to \infty$ where
\begin{equation}\label{eq1}
\sum\limits_{\kv \in I} |f_\kv^{t_m}|^2\, k^{-2 q}
\geq \constant^2 \rate^2(t_m).
\end{equation}
Recall the signum function
\begin{equation}
	\sgn x =
	\left\{
	\begin{array}{ll}
		\phantom{-}1, \quad & x \geq 0; \\
		-1, \quad & x < 0.
	\end{array}
	\right.
\end{equation}
Now notice that for each fixed time $t_m$, $\{ f^{t_m}_\kv \}_{\kv \in I}$ is a list of $|I|$ numbers in $\mathds{C}$.  Write $f^{t_m}_\kv = a_\kv^{t_m} + i b_\kv^{t_m} $ where $a_\kv^{t_m}$ and $ b_\kv^{t_m}$ are real. Then $ f^{t_m}_\kv$ is found in one of the four quadrants of the complex plane, depending on the two possibilities for $\sgn a_\kv^{t_m}$ and two possibilities for $\sgn b_\kv^{t_m}$. Thus, $\{ f^{t_m}_\kv \}_{\kv \in I}$ has $4^{|I|}$ possible states. Since we have an infinite sequence of times $t_m$, one of these states must occur infinitely many times. By taking a subsequence $t_{m_{\ell}}$, we can ensure  $\{ f^{t_{m_{\ell}} }_\kv \}_{\kv \in I}$ is the same state for all $\ell$. Let $\{ (c_\kv, d_\kv) \}_{\kv\in I}$ encode this state, meaning that~$c_\kv = \sgn a_\kv^{t_{m_{\ell} } }$ and~$d_\kv = \sgn b_\kv^{t_{m_{\ell} } }$ for all $\ell$. We see that $a_\kv^{t_{m_{\ell} } } c_\kv = \bigl|a_\kv^{t_{m_{\ell} } } \bigr|  $ and $b_\kv^{t_{m_{\ell} } } d_\kv = \bigl|b_\kv^{t_{m_{\ell} } } \bigr|  $ for all $\ell$.
Let
\begin{equation}\label{test function}
	g_\kv =
	\left\{
	\begin{array}{ll}
		\left(c_\kv + i d_\kv \right) k^{-q},\quad & \kv \in I;  \\
		0, \quad & \text{otherwise}.
		\end{array}
		\right.
\end{equation}
Notice that $g \in \Sobo^{\beta}$ because
\begin{equation}
\norm{g}_{\Sobo^{\beta}}^2 = \sum |g_\kv|^2\, k^{2\beta} = \sum\limits_{\kv \in I} (|c_\kv|^2 + |d_\kv|^2) k^{-2q} k^{ 2 \beta} = 2 \sum\limits_{\kv \in I} k^{ 2(\beta - q) } < \infty
\end{equation}
since $I$ is a finite set. We have
\begin{align*}
	\sum f_\kv^{t_{m_{\ell}}}~ \bar{g}_\kv  &= \sum\limits_{\kv \in I} \left(a_\kv^{t_{m_{\ell} } } + i b_\kv^{t_{m_{\ell} } }\right) \left(c_\kv - i d_\kv \right) ~k^{-q} \\
	&=\sum\limits_{\kv \in I} \left( a_\kv^{t_{m_{\ell} } } c_\kv  +  b_\kv^{t_{m_{\ell} } } d_\kv + i \left(b_\kv^{t_{m_{\ell} } } c_\kv - a_\kv^{t_{m_{\ell} } } d_\kv \right)  \right) ~k^{-q}\\
	&=\sum\limits_{\kv \in I} \left( \bigl|a_\kv^{t_{m_{\ell} } } \bigr|  +  \bigl|b_\kv^{t_{m_{\ell} } } \bigr| + i \left(b_\kv^{t_{m_{\ell} } } c_\kv - a_\kv^{t_{m_{\ell} } } d_\kv \right)  \right) ~k^{-q}.
\end{align*}
We conclude that
\begin{align*}
	\left| \pair{f^{t_{m_{\ell}}}}{ g} \right| &\geq \mathrm{Re} \sum f_\kv^{t_{m_{\ell}}}~ \bar{g}_\kv \\
	&=\sum\limits_{\kv \in I} \left( \bigl|a_\kv^{t_{m_{\ell} } } \bigr|  +  \bigl|b_\kv^{t_{m_{\ell} } } \bigr|   \right) ~ k^{-q}\\
	&\geq  \sum\limits_{\kv \in I}  \sqrt[]{ \bigl|a_\kv^{t_{m_{\ell} } } \bigr|^2  +  \bigl|b_\kv^{t_{m_{\ell} } } \bigr|^2  } ~ k^{-q}\\
	&=\sum\limits_{\kv \in I}\left|f_\kv^{t_{m_{\ell} } } \right|~ k^{-q},
\end{align*}
as desired.  Lastly, we use dominance of~$\ell^2$ by~$\ell^1$ together with \cref{eq1} to conclude $ \left| \pair{f^{t_{m_\ell}}}{ g} \right| \geq \constant\, \rate(t_{m_\ell})$, $\forall m_\ell$.
\end{proof}

Lemma \ref{case1} characterizes the behavior of \typeone\ functions.
We now develop the tool we need to further analyze \typetwo\ functions.

\begin{lemma}\label{intervals}
	Let $f^t$ be \typetwo\ for some positive~$h=\BigOh{ \mixnorm{f^t} }$. For any~$\delta$ with $0 < \delta < 1$, there exist sets  $I_i = \{ \kv \st J_{i-1} < |\kv| \leq J_i  \}$ and a sequence of times $T_i \to \infty$ satisfying the following:

	\begin{property}\label{intervals:property:capturesFourierMass}
		The set $I_i$ captures a significant proportion of the Fourier energy at time $T_i$, so that
		\begin{equation}\label{intervals:capturesFourierMass}
			\sum\limits_{\kv \in I_i} \left|f_\kv^{T_i} \right|^2\, k^{-2 q} \geq (1-\delta)  \norm{f^{T_i}}_{\Sobo^{-\qq}}^2\,.
		\end{equation}
	\end{property}

	\begin{property}\label{intervals:property:FourierMassDoesNotReturn}
		Enough of the Fourier energy does not return to lower frequency modes, so that
		\begin{equation}\label{intervals:FourierMassDoesNotReturn}
			\sum\limits_{|\kv| \leq J_{i-1} } \left|f_\kv^{t} \right|^2\, k^{-2 q} \leq \delta\,\rate^2(t)  \text{ for all } t \geq T_i\,.
		\end{equation}
	\end{property}
\end{lemma}

\begin{proof}
\proofsubpart{Base case} $i = 1$.  Since (by definition of absolute convergence)
\begin{equation}
	\lim\limits_{J \to \infty} \sum\limits_{|\kv| \leq J} \left|f_\kv^0 \right|^2\, k^{-2 q} =\norm{f^0}_{\Sobo^{-\qq} }^2
\end{equation}
there is a~$J_1$ such that
\begin{equation}
	\sum\limits_{|\kv| \leq J_1} \left|f_\kv^0 \right|^2\, k^{-2 q} \geq (1-\delta) \norm{f^0}_{\Sobo^{-\qq} }^2\,.
\end{equation}
We see $I_1 = \{ \kv \st J_{0} < |\kv| \leq J_1  \}$  where $J_0 = -1$ and $T_1 = 0$ therefore trivially satisfy \cref{intervals:property:capturesFourierMass,intervals:property:FourierMassDoesNotReturn}.

\proofsubpart{Induction step} Suppose that we are given $J_{i-2}, J_{i-1}$ and $T_{i-1}$ that satisfy \cref{intervals:property:capturesFourierMass,intervals:property:FourierMassDoesNotReturn}. Since $h=\BigOh{ \mixnorm{f^t} }$, there is a $c>0$ and a time $T$ so that $h(t) \leq c \mixnorm{f^{t}} ~ \forall t \geq T$. Recall that $f^t$ is \typetwo\ and so there does not exist a sequence $t_m \to \infty$ where
\begin{equation}
	\sum\limits_{|\kv| \leq J_{i-1}} \left|f_\kv^{t_m} \right|^2\, k^{-2 q}
	\geq \min( \delta, \delta/c^2) ~\rate^2(t_m), \quad\forall m.
\end{equation}
We conclude that there exist $T_i$ with~$T_i \geq T_{i-1} + 1$ and $T_i \geq T$ such that \cref{intervals:property:FourierMassDoesNotReturn} is satisfied. In particular, we have
\begin{equation}
	\sum\limits_{|\kv| \leq J_{i-1}} \left|f_\kv^{T_i} \right|^2\, k^{-2 q} < \delta /c^2 ~ \rate^2(T_i)
	\leq \delta  \norm{f^{T_i}}_{\Sobo^{-\qq}}^2\,.
\end{equation}
Hence,
\begin{equation}
	\sum\limits_{|\kv| > J_{i-1}} \left|f_\kv^{T_i} \right|^2\, k^{-2 q} > ( 1 - \delta)  \norm{f^{T_i}}_{\Sobo^{-\qq}}^2,
\end{equation}
and it follows that there is a $J_i$ large enough so that for $I_i = \{ \kv \st J_{i-1} < |\kv| < J_i  \}$ we have
\begin{equation}
	\sum\limits_{\kv \in I_i} |f_\kv^{T_i}|^2\, k^{-2 q} \geq ( 1 - \delta)  \norm{f^{T_i}}_{\Sobo^{-\qq}}^2\,
\end{equation}
which is \cref{intervals:property:capturesFourierMass}.
\end{proof}

Having developed all of the tools we will need, we now prove \cref{classify} and, in the next section, \cref{maintheoremLabel}.

\theoremOne*

\begin{proof}
The forward direction is a special case of \cref{case1} with $h(t) = \mixnorm{f^t}$. We assume $f^t$ is \tran\ and show $ \pair{f^t}{ g} = \littleOh{ \mixnorm{f^t} }$ for all $ g \in \Hq$.
We already know $ \pair{f^t}{ g} = \BigOh{ \mixnorm{f^t} }$ for all $g \in \Hq$, so
\begin{equation}
	\limsup_{t \to \infty} \frac{ \left| \pair{f^t}{ g} \right|}{\mixnorm{f^t}} < \infty, \quad\forall g \in \Hq.
\end{equation}
Seeking a contradiction, we suppose there exists a $g \in \Hq$ such that
\begin{equation}
	\limsup_{t \to \infty} \frac{\left| \pair{f^t}{ g} \right|}{\mixnorm{f^t}} = C > 0 .
\end{equation}
There is a sequence $t_n \to \infty$ such that
\begin{equation}\label{lowerbound}
	\left| \pair{f^{t_n}}{g} \right| \geq \frac{C}{2} \mixnorm{f^{t_n}}, \quad\forall n.
\end{equation}
We will show that \cref{lowerbound} implies that the Fourier coefficients of $g$ decay too slowly for $g$ to be in $\Hq$.

Since $g \in \Hq$, we can choose $\delta$ small enough that $\frac{C}{2} - 2\sqrt{\delta} \norm{g}_{\Hq} = C_0 > 0$. Applying \cref{intervals} with $\rate(t) = \mixnorm{f^t}$, there exist sets  $I_i = \{ \kv \st J_{i-1} < |\kv| \leq J_i  \}$ and a sequence of times $T_i \to \infty$ (without loss of generality, say that $\{ T_i \}$ is a subsequence of $\{ t_n \}$ above) such that we have \cref{intervals:property:capturesFourierMass,intervals:property:FourierMassDoesNotReturn}.
\Cref{intervals:property:capturesFourierMass} implies
\begin{equation}\label{e2}
	\sum\limits_{|\kv| > J_{i} } \l|f_\kv^{T_i} \r|^2\, k^{-2 q} \leq \delta \mixnorm{f^{T_i}}^{2}\,.
\end{equation}
Note that
\begin{align}
	\pair{f^{T_i}}{g} &= \sum_{j \geq 1} \sum_{\kv \in I_j} f_\kv^{T_i} ~ \bar{g}_\kv \nonumber \\
	&=  \sum_{\kv \in I_i} f_\kv^{T_i} ~ \bar{g}_\kv + E
\end{align}
where
\begin{equation}
	E = \sum_{j \neq i} \sum_{\kv \in I_j} f_\kv^{T_i} ~ \bar{g}_\kv = \sum_{j < i} \sum_{\kv \in I_j} f_\kv^{T_i} ~ \bar{g}_\kv +\sum_{j > i} \sum_{\kv \in I_j} f_\kv^{T_i} ~ \bar{g}_\kv =: E_1 + E_2\,.
\end{equation}
By the Cauchy--Schwarz inequality,
\begin{align}
	|E_1| &\leq \left( \sum_{j < i} \sum_{\kv \in I_j} \l| f_\kv^{T_i} \r|^2\, k^{-2q} \right)^{\efrac{1}{2}} \left( \sum_{j < i} \sum_{\kv \in I_j}  \l| g_\kv \r|^2\, k^{2q} \right)^{\efrac{1}{2}}.
\end{align}
Applying \cref{intervals:property:FourierMassDoesNotReturn} (with $t = T_i$), we have
\begin{equation}
	|E_1| \leq \sqrt{\delta} \mixnorm{f^{T_i}} \norm{g}_{\Hq}\,.
\end{equation}
Similarly, by Cauchy--Schwarz and \cref{e2}, we have
\begin{equation}
	|E_2| \leq \sqrt{\delta} \mixnorm{f^{T_i}} \norm{g}_{\Hq}
\end{equation}
and therefore
\begin{equation}\label{boundingerror}
	\left| \pair{f^{T_i}}{g}  \right| \leq \left| \sum_{\kv \in I_i} f_\kv^{T_i}~ \bar{g}_\kv \right| + 2\sqrt{\delta}  \mixnorm{f^{T_i}} \norm{g}_{\Hq}\,.
\end{equation}
Putting together \cref{lowerbound,boundingerror}, we have \begin{equation}\label{step}
	\left( \frac{C}{2} - 2\sqrt{\delta} \norm{g}_{\Hq} \right) \mixnorm{f^{T_i}}  \leq  \left| \sum_{\kv \in I_i} f_\kv^{T_i}~ \bar{g}_\kv \right|.
\end{equation}
Applying Cauchy--Schwarz to the right-hand side of \cref{step} we have
\begin{align}
	C_0 \mixnorm{f^{T_i}}  &\leq \left( \sum_{\kv \in I_i} \l| f_\kv^{T_i} \r|^2\, k^{-2q} \right)^{\efrac{1}{2}} \left(  \sum_{\kv \in I_i}  \l| g_\kv \r|^2\, k^{2q} \right)^{\efrac{1}{2}} \\
	&\leq \mixnorm{f^{T_i}}  \left(  \sum_{\kv \in I_i}  \l| g_\kv \r|^2\, k^{2q} \right)^{\efrac{1}{2}}.
\end{align}
Therefore
\begin{equation}
	\sum_{\kv \in I_i}  \l| g_\kv \r|^2\, k^{2q} \geq C_0^2, \quad\forall i\,.
\end{equation}
This shows that the coefficients of $g$ are large on sets $I_i$ and we have
\begin{equation}
	\norm{g}_{\Hq}^2 = \sum_{i} \sum_{\kv \in I_i}  \l| g_\kv \r|^2\, k^{2q} \geq \sum_i C_0^2 = \infty.
\end{equation}
We conclude that $g$ is not in $\Hq$ --- a contradiction.
\end{proof}

\section{Proof of Theorem 2}\label{proofs2}

\maintheorem*
\begin{proof} [Proof of \cref{maintheoremLabel}]
If $f^t$ is \typeone, then we are done by \cref{case1}, so say $f^t$ is \typetwo.
Take some $\delta < 1/3$ and apply \cref{intervals} to construct sets $\l\{ I_i \r\}_{i=1}^\infty$ and a sequence $\l\{ T_i \r\}_{i = 1}^\infty$.
Let $\l\{ T_{i_{\ell}} \r\}_{\ell = 1}^\infty$ be a subsequence of $\l\{ T_i \r\}_{i = 1}^\infty$ satisfying
\begin{equation}\label{conditionthree}
	\sum\limits_{\ell > L} \ratio{i_\ell} \leq \delta^2\, \ratio{i_L}.
\end{equation}
and
\begin{equation}\label{eqn2}
	\sum\limits \ratio{i_\ell} \leq \delta.
\end{equation}
This can be done since $\rate(t) = o(\mixnorm{f^t})$. Let $g$ be the function with Fourier coefficients given by
\begin{equation}\label{thm:coeff}
g_\kv =
\left\{
\begin{array}{ll}
f_\kv^{T_{i_{\ell}}}  k^{-2q} \mixnorm{f^{T_{i_{\ell}}}}^{- 2} h(T_{i_\ell}) & \kv \in I_{i_{\ell}} \,; \\
0 & \text{otherwise}. \\
\end{array}
\right.
\end{equation}
\cref{eqn2} allows us to conclude that $g \in \Hq$:
\begin{align*}
\norm{g}_{\Hq}^2 &= \sum \l|g_\kv \r|^2\, k^{2 q}\\
&=\sum\limits_{i_{\ell}} \sum\limits_{\kv \in I_{i_{\ell}}} \l|f_\kv^{T_{i_{\ell}}} \r|^2\, k^{-4q} \mixnorm{f^{T_{i_{\ell}}}}^{ - 4} h^2(T_{i_{\ell}}) ~k^{2 q}\\
&= \sum\limits_{i_{\ell}} \mixnorm{f^{T_{i_{\ell}}}}^{ - 4} h^2(T_{i_{\ell}}) \l( \sum\limits_{\kv \in I_{i_{\ell}}} \l|f_\kv^{T_{i_{\ell}}} \r|^2\, k^{-2q} \r)   \\ \label{eqn:PartOne}
&\leq \sum\limits_{i_{\ell}} \mixnorm{f^{T_{i_{\ell}}}}^{ - 2} h^2(T_{i_{\ell}}).
\end{align*}

We will now finish the proof by showing $	\left| \pair{f^{T_{i_\ell}}}{g} \right| \geq (1-3\delta) ~h(T_{i_\ell})$.
We begin with some notation. Split the following sum into two parts:
\begin{align*}
	\pair{ f^{T_{i_{\ell}}} }{g} &= \sum\limits_{j_{\ell}} \sum\limits_{\kv \in I_{j_{\ell}}}  f^{T_{i_{\ell}}}_\kv ~ \overline{ f_\kv^{T_{j_{\ell}}}  } k^{-2q} \mixnorm{f^{T_{j_{\ell}}}}^{ - 2} \rate(T_{j_{\ell}})\\
	&= S^{T_{i_{\ell}}} + E^{T_{i_{\ell}}},
\end{align*}
where $S^{T_{i_{\ell}}}$ is the sum when $j_\ell = i_\ell$\,:
\begin{align*}
	S^{T_{i_{\ell}}} &= \sum\limits_{j_{\ell} = i_{\ell}} \sum\limits_{\kv \in I_{i_{\ell}}} f^{T_{i_{\ell}}}_\kv ~ \overline{ f_\kv^{T_{j_{\ell}}}  } k^{-2q} \mixnorm{f^{T_{j_{\ell}}}}^{ - 2} \rate(T_{j_{\ell}})\\
	&= \sum\limits_{\kv \in I_{i_{\ell}}}  \l|f^{T_{i_{\ell}}}_\kv \r|^2 ~ k^{-2q} ~ \mixnorm{f^{T_{i_{\ell}}}}^{ - 2} ~ \rate(T_{i_{\ell}}),
\end{align*}
and $E^{T_{i_{\ell}}}$ is the sum over $j_\ell \neq i_\ell$\,:
\begin{equation}
	E^{T_{i_{\ell}}} = 	\sum\limits_{j_{\ell} \ne i_{\ell}} \sum\limits_{\kv \in I_{j_{\ell}}}  f^{T_{i_{\ell}}}_\kv ~ \overline{ f_\kv^{T_{j_{\ell}}}  } k^{-2q} \mixnorm{f^{T_{j_{\ell}}}}^{ - 2} \rate(T_{j_{\ell}})\,.
        \label{eq:ETil}
\end{equation}
The idea is that $g_\kv$ is constructed to agree well with $f^{T_{i_{\ell}}}_\kv$  when $\kv \in I_{i_{\ell}}$. We will show that $S^{T_{i_{\ell}}}$ dominates the error $E^{T_{i_{\ell}}}$. Consider $j_\ell < i_\ell$ and $j_\ell > i_\ell$ separately in~\cref{eq:ETil}; taking absolute value, we have

\begin{equation}
	\l|E^{T_{i_{\ell}}} \r| \leq \sum\limits_{j_{\ell} \ne i_{\ell}} \sum\limits_{\kv \in I_{j_{\ell}}}  \l|f^{T_{i_{\ell}}}_\kv \r |\l|f^{T_{j_{\ell}}}_\kv \r| ~ k^{-2q} ~ \mixnorm{f^{T_{j_{\ell}}}}^{ - 2} ~ \rate(T_{j_{\ell}})
\end{equation}
and let $E^{T_{i_{\ell}}}_1$ be the sum over $j_\ell < i_\ell$\,:
\begin{equation}
	E^{T_{i_{\ell}}}_1 := \sum\limits_{j_{\ell} < i_{\ell}} \sum\limits_{\kv \in I_{j_{\ell}}} \l|f^{T_{i_{\ell}}}_\kv \r |\l|f^{T_{j_{\ell}}}_\kv \r| ~ k^{-2q} ~ \mixnorm{f^{T_{j_{\ell}}}}^{ - 2} ~ \rate(T_{j_{\ell}}) \,.
\end{equation}
Similarly define $E^{T_{i_{\ell}}}_2$ to be the sum over $j_\ell > 	i_\ell$\,:
\begin{equation}
	E^{T_{i_{\ell}}}_2 := \sum\limits_{j_{\ell} > i_{\ell}} \sum\limits_{\kv \in I_{j_{\ell}}} \l|f^{T_{i_{\ell}}}_\kv \r |\l|f^{T_{j_{\ell}}}_\kv \r| ~ k^{-2q} ~ \mixnorm{f^{T_{j_{\ell}}}}^{ - 2} ~ \rate(T_{j_{\ell}}) \,.
\end{equation}
We now bound $E^{T_{i_{\ell}}}_1$ using the Cauchy--Schwarz inequality:
\begin{align*}
	E^{T_{i_{\ell}}}_1 &\leq \left( \sum\limits_{j_{\ell} < i_{\ell}} \sum\limits_{\kv \in I_{j_{\ell}}}   \left|f_\kv^{ T_{i_{\ell}} } \right|^2  k^{-2q}    \right)^{\efrac{1}{2}}    \left( \sum\limits_{j_{\ell} < i_{\ell}} \sum\limits_{\kv \in I_{j_{\ell}}}  \left|f_\kv^{T_{j_{\ell}}} \right|^2  k^{-2q} \mixnorm{f^{T_{j_{\ell}}}}^{ - 4} ~ \rate^2(T_{j_{\ell}}) \right)^{\efrac{1}{2}} \\
	&\leq \left( \sum\limits_{j_{\ell} < i_{\ell}} \sum\limits_{\kv \in I_{j_{\ell}}}   \left|f_\kv^{ T_{i_{\ell}} } \right|^2  k^{-2q}    \right)^{\efrac{1}{2}}    \left( \sum\limits_{j_{\ell} < i_{\ell}}  \mixnorm{f^{T_{j_{\ell}}}}^{ - 2} ~ \rate^2(T_{j_{\ell}}) \right)^{\efrac{1}{2}} .
\end{align*}
We use \cref{intervals:property:FourierMassDoesNotReturn} from \cref{intervals} to bound the first factor and \cref{eqn2} to bound the second factor:
\begin{equation}
	E^{T_{i_{\ell}}}_1 \leq \left( \delta  \rate^2(T_{i_{\ell}} )  \right)^{\efrac{1}{2}}    \left(\delta \right)^{\efrac{1}{2}}
	= \delta \,  h(T_{i_{\ell}}).
\end{equation}
We similarly bound $E^{T_{i_{\ell}}}_2$:
\begin{align*}
	E^{T_{i_{\ell}}}_2 &\leq \left( \sum\limits_{j_{\ell} > i_{\ell}} \sum\limits_{\kv \in I_{j_{\ell}}}   \left|f_\kv^{ T_{i_{\ell}} } \right|^2  k^{-2q}    \right)^{\efrac{1}{2}}    \left( \sum\limits_{j_{\ell} > i_{\ell}} \sum\limits_{\kv \in I_{j_{\ell}}}  \left|f_\kv^{T_{j_{\ell}}} \right|^2  k^{-2q} \mixnorm{f^{T_{j_{\ell}}}}^{ - 4} ~ \rate^2(T_{j_{\ell}}) \right)^{\efrac{1}{2}} \\
	&\leq \mixnorm{f^{T_{i_{\ell}}}}   \left( \sum\limits_{j_{\ell} > i_{\ell}}  \mixnorm{f^{T_{j_{\ell}}}}^{ - 2} ~ \rate^2(T_{j_{\ell}}) \right)^{\efrac{1}{2}} .
\end{align*}
Using \cref{conditionthree}, we find
\begin{equation}
  E^{T_{i_{\ell}}}_2 \leq \mixnorm{f^{T_{i_{\ell}}}} \l( \delta^2 \l( \frac{\rate(T_{i_{\ell}})}{\mixnorm{f^{T_{i_{\ell}}}}}	\r)^2	\r)^{1/2}
	= \delta ~  h(T_{i_{\ell}})
\end{equation}
and therefore
\begin{align*}
	\pair{ f^{T_{i_{\ell}}} }{g} &=   S^{T_{i_{\ell}}} +  E^{T_{i_{\ell}}}   \\
	&\geq  S^{T_{i_{\ell}}} -  E^{T_{i_{\ell}}}_1 -   E^{T_{i_{\ell}}}_2  \\
	&\geq S^{T_{i_{\ell}}} -  \delta  ~\rate(T_{i_\ell} )  -   \delta  ~\rate(T_{i_\ell} ).
\end{align*}
Again using \cref{intervals:property:capturesFourierMass} from \cref{intervals} that the set $I_{i_{\ell}}$ captures a large proportion of the Sobolev norm, we conclude
\begin{align*}
	\pair{ f^{T_{i_{\ell}}} }{g} &\geq ( 1 - \delta)  \mixnorm{f^{T_{i_\ell}}}^2  \mixnorm{f^{T_{i_{\ell}}}}^{ -2 } ~ \rate(T_{i_{\ell}})  -  2\delta  ~\rate(T_{i_\ell} ) \\
	&= \left( 1- 3\delta \right) ~\rate(T_{i_\ell} ).
\end{align*}
\end{proof}

\section*{Acknowledgments}

The authors thank Gautam Iyer for helpful discussions and Georg Gottwald for asking the question that prompted this research.
This work was supported in part by NSF Award DMS-1813003.

\setcitestyle{numbers}
\bibliographystyle{plainnat}
\bibliography{journals_abbrev,mixconv}

\end{document}